\documentclass[10pt,a4paper]{article}

\usepackage{amsfonts,amsgen,amsmath,amstext,amsbsy,amsopn,amsthm,cases,listings
}
\usepackage{multirow}
\usepackage{graphicx}
\usepackage{epstopdf}
\usepackage{pgf,tikz}
\usepackage{mathrsfs}
\usepackage[marginal]{footmisc}
\usepackage{enumitem}
\usepackage[titletoc]{appendix}
\usepackage{booktabs}
\usepackage{url}
\usepackage{mathtools}
\usepackage{pgfplots}
\usepackage{authblk}
\usepackage{amssymb}
\usepackage{wasysym}
\usepackage{empheq}
\usepackage{dsfont}
\usepackage{tikz}
\usepackage{longtable}
\usepackage{float}
\usepackage{subfigure}
\pgfplotsset{compat=1.18}
\usepackage{mathrsfs}
\usepackage{wasysym}
\InputIfFileExists{uwasy.fd}{}{}
\DeclareFontShape{U}{wasy}{b}{n}{<5><6><7><8><9> gen * wasy <10> <10.95> <12> <14.4> <17.28> <20.74> <24.88> wasyb10}{}
\usetikzlibrary{arrows}
\usepackage{aligned-overset}
\usepackage{bm}
\usepackage{bbm}
\usepackage[T1]{fontenc}
\usepackage[backref=page]{hyperref}
\usepackage[a4paper,margin=2.5cm]{geometry}
\usepackage[thinc]{esdiff}

\allowdisplaybreaks[1]

\definecolor{uuuuuu}{rgb}{0.27,0.27,0.27}
\definecolor{sqsqsq}{rgb}{0.1255,0.1255,0.1255}

\newtheorem{definition}{Definition}[section]
\newtheorem{theorem}[definition]{Theorem}
\newtheorem{lemma}[definition]{Lemma}

\newtheorem{claim}[definition]{Claim}

\newcommand{\coex}{\operatorname{co}\!\operatorname{ex}}

\newcommand{\exlink}{\operatorname{ex}_{\rm link}}
\newcommand{\pilink}{\pi_{\rm link}}

\makeatletter
\newsavebox\myboxA
\newsavebox\myboxB
\newlength\mylenA
\newcommand*\xoverline[2][0.81]{%
    \sbox{\myboxA}{$\m@th#2$}%
    \setbox\myboxB\null%
    \ht\myboxB=\ht\myboxA%
    \dp\myboxB=\dp\myboxA%
    \wd\myboxB=#1\wd\myboxA%
    \sbox\myboxB{$\m@th\overline{\copy\myboxB}$}%
    \setlength\mylenA{\the\wd\myboxA}%
    \addtolength\mylenA{-\the\wd\myboxB}%
    \ifdim\wd\myboxB<\wd\myboxA%
       \rlap{\hskip 0.5\mylenA\usebox\myboxB}{\usebox\myboxA}%
    \else
        \hskip -0.5\mylenA\rlap{\usebox\myboxA}{\hskip 0.5\mylenA\usebox\myboxB}%
    \fi}
\makeatother
\begin{document}
\title{\bf\Large A note on the $t$-partite link problem of F\"uredi}
\date{\today}
\author[1]{Jianfeng Hou\thanks{Email: \texttt{jfhou@fzu.edu.cn}}}
\author[2,3]{Xinmin Hou\thanks{Email: \texttt{xmhou@ustc.edu.cn}}}
\author[2]{Xizhi Liu\thanks{Email: \texttt{liuxizhi@ustc.edu.cn}}}
\author[ ]{Jiasheng Zeng\thanks{Email: \texttt{jasonzeng@mail.ustc.edu.cn}}}
\author[1]{Yixiao Zhang\thanks{Email: \texttt{fzuzyx@gmail.com}}}
\affil[1]{\small Center for Discrete Mathematics, Fuzhou University, Fuzhou, China}
\affil[2]{\small School of Mathematical Sciences, University of Science and Technology of China, Hefei, Anhui, China}
\affil[3]{\small Hefei National Laboratory, University of Science and Technology of China, Hefei, Anhui, China}
\maketitle
\begin{abstract}
    Motivated by the Erd\H{o}s--S\'{o}s bipartite link conjecture, F\"{u}redi (Oberwolfach, 2004) asked for the asymptotic maximum edge density $\pi_{\mathrm{link}}(t)$ of $3$-graphs in which the link graph of every vertex is $t$-partite. Goldwasser's recursive blow-up construction based on projective planes gives the lower bound $\pi_{\mathrm{link}}(t)\ge 1-t^{-1}-(2+o_t(1))t^{-2}$ whenever $t-1$ is a prime power. In this note, we prove the upper bound $\pi_{\mathrm{link}}(t)\le 1-t^{-1}-t^{-2}/12$ for every $t \ge 2$. Together with Goldwasser's construction, this determines, up to a constant factor, the correct order of the gap between $\pi_{\mathrm{link}}(t)$ and the trivial averaging upper bound $1-t^{-1}$ for all prime-power values of $t-1$.

    In fact, our argument applies in the more general setting of $3$-graphs with no generalized daisies, equivalently, $3$-graphs in which the link graph of every vertex is $K_{t+1}$-free. We also establish an analogous upper bound for the positive $(r-1)$-codegree Tur\'an density of generalized daisies.
\end{abstract}
\section{Introduction}\label{SEC:Introduction}

Given an integer $r\ge 2$, an $r$-uniform hypergraph, or simply an $r$-graph, is a collection of $r$-subsets of a vertex set. We usually identify a hypergraph with its edge set and write $|\mathcal H|$ for the number of its edges.

For a $3$-graph $\mathcal H$ and a vertex $v\in V(\mathcal H)$, the \emph{link graph} of $v$ is
\[
    L_{\mathcal H}(v):=
    \left\{xy\in \binom{V(\mathcal H)\setminus\{v\}}2:
    vxy\in \mathcal H\right\}.
\]
An old conjecture of Erd\H{o}s and S\'os, see \cite[p.~328]{FF84}, asserts that if every vertex link of an $n$-vertex $3$-graph is bipartite, then
\[
    |\mathcal H|\le \frac14\binom n3+o(n^3).
\]
This conjecture remains open. One difficulty is that, even if the conjectured bound is correct, there are several rather different constructions that are asymptotically extremal; see \cite{FF84,MPS11,FV13} for further discussion. Extending this conjecture, F\"uredi (Oberwolfach 2004, see \cite[Problem~14]{MPS11}) asked for the corresponding asymptotic problem when every vertex link has chromatic number at most $t$.

Define
\[
    \exlink(n,t):=
    \max\left\{
    |\mathcal H|:
    |V(\mathcal H)|=n,\,
    \chi(L_{\mathcal H}(v))\le t
    \text{ for every }v\in V(\mathcal H)
    \right\},
\]
where $\chi(L_{\mathcal H}(v))$ denotes the chromatic number of the link graph of $v$. Let
\[
    \pilink(t):=
    \lim_{n\to\infty}\frac{\exlink(n,t)}{\binom n3}.
\]

Goldwasser's recursive blow-up construction based on projective planes gives the following lower bound for F\"uredi's $t$-partite link problem; see Section~\ref{SEC:LowerBounds} for details.

\begin{theorem}[Goldwasser]\label{THM:t_partite_lower}
    Suppose that $t-1$ is a prime power. Then
    \[
        \pilink(t)
        \ge \frac{(t-1)^2}{t^2-t+2}
        =1-\frac1t-\frac{2+o(1)}{t^2}.
    \]
\end{theorem}

Our first main result is the following upper bound.

\begin{theorem}\label{THM:t_partite_upper}
    For every integer $t\ge2$,
    \[
        \pilink(t)
        \le 1-\frac1t-\frac{1}{12t^2}.
    \]
\end{theorem}

Together with Goldwasser's construction, Theorem~\ref{THM:t_partite_upper} determines, up to a constant factor, the correct order of the gap between $\pilink(t)$ and the trivial averaging upper bound $1-1/t$ along the infinite sequence of values for which $t-1$ is a prime power.

We prove Theorem~\ref{THM:t_partite_upper} through a more general hypergraph Tur\'an problem. Given a family $\mathcal F$ of $r$-graphs, an $r$-graph $\mathcal H$ is $\mathcal F$-free if it contains no member of $\mathcal F$ as a subgraph. The Tur\'an number $\operatorname{ex}(n,\mathcal F)$ is the maximum number of edges in an $\mathcal F$-free $r$-graph on $n$ vertices, and the Tur\'an density is
\[
    \pi(\mathcal F):=
    \lim_{n\to\infty}
    \frac{\operatorname{ex}(n,\mathcal F)}{\binom n r}.
\]
The existence of this limit follows from the standard averaging argument of Katona--Nemetz--Simonovits. The study of Tur\'an numbers is one of the central themes of extremal combinatorics. For graphs, Tur\'an's theorem~\cite{Tur41} and the Erd\H{o}s--Stone--Simonovits theorem~\cite{ES46,ES66} determine all Tur\'an densities. By contrast, for $r\ge3$, determining hypergraph Tur\'an densities remains notoriously difficult; see the survey of Keevash~\cite{Kee11} for background.

For a graph $F$, the \emph{$r$-suspension} $\hat F^r$ is obtained by adjoining a fixed set $S$ of $r-2$ new vertices to every edge of $F$. The \emph{$(r,t)$-daisy} $\mathcal D_{r,t}$, introduced by Bollob\'as--Leader--Malvenuto~\cite{BLM11}, is the $r$-suspension of $K_t$. Daisies are closely connected with several other extremal problems. One source of motivation comes from Tur\'an-type questions in the hypercube, where related forbidden configurations have been studied in a number of works; see, for example, \cite{Kos76,JE89,AKS07,JT10,Alo24,EIL25}. Another source comes from extremal problems for set families with forbidden subposets, where suspension-type and layered configurations arise naturally; see \cite{Buk09,GL16}. Thus the daisy problem may be viewed as a hypergraph Tur\'an problem that also reflects phenomena from the Boolean lattice and the hypercube.

Bollob\'as--Leader--Malvenuto~\cite{BLM11} conjectured that, for each fixed $t\ge4$, the density $\pi(\mathcal D_{r,t})$ tends to zero as $r\to\infty$. This conjecture was disproved by Ellis--Ivan--Leader~\cite{EIL24}, who proved, among other results, that when $t-1$ is a prime power,
\[
    \lim_{r\to\infty}\pi(\mathcal D_{r,t+1})
    \ge
    \prod_{i=1}^{\infty}
    \left(1-\frac{1}{(t-1)^i}\right)
    =
    1-\frac1t-\frac{2+o(1)}{t^2}.
\]
In particular, by the theorem of Baker--Harman--Pintz on gaps between consecutive primes~\cite{BHP01}, one obtains
\[
    \lim_{r\to\infty}\pi(\mathcal D_{r,t+1})
    \ge
    1-\frac1t-O(t^{-1.475})
\]
for all $t$. The exponent $1.475$ comes from the current unconditional prime-gap estimate; under Cram\'er's conjecture~\cite{Cra36}, it would improve to $2-o(1)$.

On the upper-bound side, Hou--Liu--Zhang~\cite{HLZ26DaisyOld} previously proved the first nontrivial estimate of the form
\[
    \pi(\mathcal D_{r,t+1})
    \le
    1-\frac1t-\Omega\left(\frac{1}{t^8}\right).
\]
Our main Tur\'an result improves this to the best possible order of magnitude.

\begin{theorem}\label{THM:daisy_Turan_upper_bound}
    For every $r\ge3$ and $t\ge2$, we have
    \[
        \pi(\mathcal D_{r,t+1})
        \le 1-\frac1t-\frac{1}{12t^2}.
    \]
\end{theorem}

For $r=3$, being $\mathcal D_{3,t+1}$-free is equivalent to requiring every vertex link to be $K_{t+1}$-free. Thus Theorem~\ref{THM:daisy_Turan_upper_bound} is stronger than Theorem~\ref{THM:t_partite_upper}, since every $t$-partite graph is $K_{t+1}$-free.

Let us briefly compare our approach with the previous proof in \cite{HLZ26DaisyOld}. The argument in \cite{HLZ26DaisyOld} compares different link partitions using the Birkhoff--von Neumann theorem on doubly stochastic matrices~\cite{Bir46,JDM60}, whose standard proof is closely related to Hall's theorem~\cite{Hal35}. In contrast, the present proof of Theorem~\ref{THM:daisy_Turan_upper_bound} uses a weighted averaging argument over link partitions. Roughly speaking, after approximating each $K_{t+1}$-free link by a $t$-partite graph using a theorem of F\"uredi, we average the missing and bad degrees over a carefully chosen base vertex and a carefully chosen pair of parts. This produces a vertex $a_0$ and a set $B$ such that $L(a_0)[B]$ has more edges than allowed by Tur\'an's theorem.

Although Theorem~\ref{THM:daisy_Turan_upper_bound}, together with the projective-plane constructions, determines the correct order of the gap from the classical Tur\'an bound $1-1/t$, determining the exact constant appears to be substantially more delicate. Indeed, even the first nontrivial case is closely related to classical open problems on $3$-graphs. When $(r,t)=(3,2)$, the forbidden daisy $\mathcal D_{3,3}$ is the $3$-graph $K_4^{3-}$ obtained from $K_4^3$ by deleting one edge. The determination of $\pi(K_4^{3-})$ has been studied for decades; see, for example, de Caen~\cite{Cae83}, Frankl--F\"uredi~\cite{FF84}, Mubayi~\cite{Mu03}, Baber--Talbot~\cite{BT11}, and Falgas-Ravry--Vaughan~\cite{FV13}. This smallest case already illustrates why one should not expect the exact value of $\pi(\mathcal D_{r,t+1})$ to follow from a sharp-order argument alone.

We also study an analogous positive codegree problem. Let $\mathcal H$ be a nonempty $r$-graph. The \emph{positive $(r-1)$-codegree} of $\mathcal H$ is
\[
    \delta_{r-1}^+(\mathcal H):=
    \min\left\{d_{\mathcal H}(S):
    S\in\binom{V(\mathcal H)}{r-1},\ d_{\mathcal H}(S)>0\right\},
\]
where $d_{\mathcal H}(S)$ denotes the number of edges containing the $(r-1)$-set $S$. For an $r$-graph $F$, define
\[
    \coex_{r-1}^+(n,F):=
    \max\{\delta_{r-1}^+(\mathcal H):
    |V(\mathcal H)|=n,\,
    \mathcal H\text{ is }F\text{-free}\}.
\]
The positive $(r-1)$-codegree density is
\[
    \gamma_{r-1}^+(F):=
    \lim_{n\to\infty}
    \frac{\coex_{r-1}^+(n,F)}n.
\]
The existence of this limit was established by Pikhurko~\cite{Pik23}, and several first examples were studied by Halfpap--Lemons--Palmer~\cite{HLP24}; further jumps and exact values were obtained by Balogh--Halfpap--Lidick\'y--Palmer~\cite{BHLP24}.

We determine the order of the positive $(r-1)$-codegree density of $\mathcal D_{r,t+1}$. The lower bound comes from a linear-algebraic construction of Ellis--Ivan--Leader; when $r=3$, this reduces to the non-recursive projective-plane blow-up.

\begin{theorem}\label{THM:positive_lower_general_r}
Let $r\ge 3$ and $t\ge3$, and suppose that $t-1$ is a prime power. Then
\[
    \gamma_{r-1}^+(\mathcal D_{r,t+1})
    \ge
    \frac{(t-1)^r-(t-1)^{r-1}}{(t-1)^r-1}
    = 1-\frac1t-\frac1{t^2}+O(t^{-3}).
\]
\end{theorem}

Our corresponding upper bound follows from Theorem~\ref{THM:t_partite_upper}, the Andr\'asfai--Erd\H{o}s--S\'os degree theorem, and a standard edge-counting consequence of positive codegree.

\begin{theorem}\label{THM:positive_upper}
    For every $r\ge 3$ and $t\ge 2$,
    \[
        \gamma_{r-1}^+(\mathcal D_{r,t+1})
        \le
        1-\frac1t-\frac{1}{36t^2}.
    \]
\end{theorem}

Together, Theorems~\ref{THM:positive_lower_general_r} and~\ref{THM:positive_upper} show that, whenever $t-1$ is a prime power,
\[
    \gamma_{r-1}^+(\mathcal D_{r,t+1})
    =
    1-\frac1t-\Theta(t^{-2})
\]
for every fixed $r\ge3$.


\paragraph{Organization of the paper.}
In Section~\ref{SEC:LowerBounds}, we record the projective-plane construction for Theorem~\ref{THM:t_partite_lower} and the linear-algebraic construction for Theorem~\ref{THM:positive_lower_general_r}. In Section~\ref{SEC:MainProof}, we prove the main daisy Tur\'an upper bound, Theorem~\ref{THM:daisy_Turan_upper_bound}, which implies Theorem~\ref{THM:t_partite_upper}. Finally, in Section~\ref{SEC:PositiveCodegree}, we prove the positive codegree upper bound, Theorem~\ref{THM:positive_upper}.

\section{Lower bound constructions}\label{SEC:LowerBounds}
In this section we record the projective-plane construction used in Theorem~\ref{THM:t_partite_lower}, and the linear-algebraic construction used in Theorem~\ref{THM:positive_lower_general_r}.

Let $q\ge2$ be a prime power, and let $PG(2,q)$ be the projective plane of order $q$. It has
\[
    m:=q^2+q+1
\]
points and the same number of lines. Every line contains exactly $q+1$ points, every point lies on exactly $q+1$ lines, and every pair of distinct points lies on a unique line.

Let $\mathcal P_q$ be the $3$-graph whose vertices are the points of $PG(2,q)$ and whose edges are the triples of non-collinear points. Then
\begin{align*}
    |\mathcal P_q|
    &= \binom m3-m\binom{q+1}{3}
      =\frac16 q^3(q+1)(q^2+q+1)
      =\frac16q^3(q+1)m.
\end{align*}

We first take a balanced blow-up of $\mathcal P_q$, and then recursively repeat the same construction inside each vertex class. Let $E(n)$ denote the resulting number of edges on $n$ vertices, ignoring divisibility errors. Then
\[
    E(n)=|\mathcal P_q|\left(\frac nm\right)^3+mE\left(\frac nm\right)+o(n^3).
\]
Thus, if $E(n)=(c+o(1))n^3$, then
\[
    c=\frac{|\mathcal P_q|}{m^3}+\frac{c}{m^2},
\]
and hence
\begin{align*}
    \frac{E(n)}{\binom n3}
    =\frac{6|\mathcal P_q|}{m^3}\cdot\frac1{1-1/m^2}+o(1)
    =\frac{q^3(q+1)}{m^2-1}+o(1)
      =\frac{q^2}{q^2+q+2}+o(1).
\end{align*}

The construction is $\mathcal D_{3,q+2}$-free, and in fact every vertex link is $(q+1)$-partite. Indeed, for a point $p$ of the projective plane, the $q+1$ lines through $p$ partition all other points; two points are adjacent in the link of $p$ precisely when they lie on different lines through $p$. The recursive construction preserves this link-colouring property.

Taking $q=t-1$, we obtain Theorem~\ref{THM:t_partite_lower}.





\bigskip

Next we present the construction for Theorem~\ref{THM:positive_lower_general_r}, which is a slight variation of the construction by Ellis--Ivan--Leader~\cite{EIL24}.

Let $q=t-1$ be a prime power, and let $X:=\mathbb F_q^r\setminus\{0\}$. Define an $r$-graph $\mathcal B_{r,q}$ on $X$ by declaring
\[
    \{x_1,\ldots,x_r\}\in \mathcal B_{r,q}
    \quad\Longleftrightarrow\quad
    x_1,\ldots,x_r
    \text{ are linearly independent over }\mathbb F_q .
\]
We first note that $\mathcal B_{r,q}$ has the stronger property that the link graph of every $(r-2)$-set is $(q+1)$-partite. Indeed, let $R$ be an $(r-2)$-set. If the vectors in $R$ are linearly dependent, then $L(R)$ is empty. Otherwise, let $W=\operatorname{span}(R)$. Then $\mathbb F_q^r/W$ has dimension $2$, and its one-dimensional subspaces give $q+1$ directions. Assign each vector $u\notin W$ to the direction spanned by $u+W$, and assign vectors in $W$ arbitrarily. Two vertices in the same direction cannot extend $R$ to a basis, so no link edge lies inside one part. Thus $L(R)$ is $(q+1)$-partite.

Taking $q=t-1$, every $(r-2)$-link is $t$-partite, and hence $\mathcal B_{r,q}$ is $\mathcal D_{r,t+1}$-free.

Now take a balanced blow-up of $\mathcal B_{r,q}$. Thus each vector $x\in X$ is replaced by a vertex class $V_x$ of size $N$, and we insert all $r$-sets crossing classes $V_{x_1},\ldots,V_{x_r}$ whenever $\{x_1,\ldots,x_r\}$ is an edge of $\mathcal B_{r,q}$. Denote the resulting $r$-graph by $\mathcal H_N$. Since $\mathcal D_{r,t+1}$ is covering, any copy of $\mathcal D_{r,t+1}$ in the blow-up would project to a copy of $\mathcal D_{r,t+1}$ in $\mathcal B_{r,q}$. Hence $\mathcal H_N$ is $\mathcal D_{r,t+1}$-free.

We compute its positive $(r-1)$-codegree. Let $A$ be an $(r-1)$-set of vertices of $\mathcal H_N$ with positive degree. Then $A$ must meet $r-1$ distinct vertex classes, say
\[
    V_{x_1},\ldots,V_{x_{r-1}},
\]
where $x_1,\ldots,x_{r-1}$ are linearly independent. A vertex in a class $V_y$ extends $A$ to an edge exactly when
\[
    y\notin \operatorname{span}(x_1,\ldots,x_{r-1}).
\]
The span has $q^{r-1}$ vectors, including $0$, and therefore the number of nonzero vectors outside it is
\[
    q^r-q^{r-1}.
\]
Each corresponding class has size $N$. Hence
\[
    d_{\mathcal H_N}(A)=(q^r-q^{r-1})N.
\]
All other $(r-1)$-sets have degree either $0$ or at least this value. Since
\[
    |V(\mathcal H_N)|=(q^r-1)N,
\]
we obtain
\[
    \frac{\delta_{r-1}^+(\mathcal H_N)}{|V(\mathcal H_N)|}
    =
    \frac{q^r-q^{r-1}}{q^r-1}.
\]
Letting $N\to\infty$ proves Theorem~\ref{THM:positive_lower_general_r}.

\section{Proof of Theorem~\ref{THM:daisy_Turan_upper_bound}}\label{SEC:MainProof}
This section proves Theorem~\ref{THM:daisy_Turan_upper_bound}.  The main point is a weighted averaging argument over the missing and bad degrees in the vertex links.  The proof is given for $r=3$; the reduction to larger uniformities is the standard link-averaging argument.

We need the following two standard tools.  The first is the extremal cloning lemma for blow-up-invariant forbidden hypergraphs, used also in~\cite{HLZ26DaisyOld}.

\begin{lemma}\label{LEM:degree_control}
    Let $t\ge2$, and let $\mathcal H$ be an extremal $\mathcal D_{3,t+1}$-free $3$-graph on $n$ vertices. Then
    \[
        \Delta(\mathcal H)-\delta(\mathcal H)\le n-2.
    \]
\end{lemma}

The second tool is F\"uredi's stability theorem for $K_{t+1}$-free graphs.

\begin{theorem}[\cite{Furedi15}]\label{THM:Furedi_linear_stability_new}
    Let $N\ge t\ge2$, and let $G$ be an $N$-vertex $K_{t+1}$-free graph with $e(T_{N,t})-p$ edges. Then $G$ can be made $t$-partite by deleting at most $p$ edges.
\end{theorem}

\begin{proof}[Proof of Theorem~\ref{THM:daisy_Turan_upper_bound}]
    We first reduce to the case $r=3$.  If an $(r+1)$-graph $\mathcal G$ is $\mathcal D_{r+1,t+1}$-free, then every vertex link of $\mathcal G$ is $\mathcal D_{r,t+1}$-free.  Hence
    \[
        |\mathcal G|
        =\frac1{r+1}\sum_{v\in V(\mathcal G)} |L_{\mathcal G}(v)|
        \le \frac{n}{r+1}\mathrm{ex}(n-1,\mathcal D_{r,t+1}).
    \]
    Dividing by $\binom n{r+1}$ and letting $n\to\infty$ gives
    \[
        \pi(\mathcal D_{r+1,t+1})\le \pi(\mathcal D_{r,t+1}).
    \]
    Thus it is enough to prove the theorem for $r=3$.

    The case $t=2$ follows from the known bound $\pi(K_4^{3-})<0.287$, which is stronger than $1/2-1/48$; see, for example,~\cite{BT11,FV13}.  We therefore assume throughout the proof that $t\ge3$.

    Put $\varepsilon:=\frac{1}{12t^2}$ and set $\theta:=10^{-4}$. Then
    \[
        (1+\theta)\frac{34}{36}<1,
    \]

    Suppose for contradiction that, for infinitely many $n$, there exists an $n$-vertex $\mathcal D_{3,t+1}$-free $3$-graph with at least $\left(1-\frac1t-\varepsilon\right)\binom n3$ edges.  Choose such a graph $\mathcal H$ extremal among all $\mathcal D_{3,t+1}$-free $3$-graphs on $n$ vertices.

    Since $\mathcal D_{3,t+1}$ is the $3$-suspension of $K_{t+1}$, every link $L(x)$ is $K_{t+1}$-free.  By Lemma~\ref{LEM:degree_control}, for all sufficiently large $n$ depending on $t$,
    \begin{align}\label{EQ:min_degree_main_revised}
        \delta(\mathcal H)
        &\ge \frac{3|\mathcal H|}{n}-n
        \ge
        \frac12\left(1-\frac1t-(1+\theta)\varepsilon\right)n^2 .
    \end{align}
    We shall assume from now on that $n$ is large enough for all lower-order terms in the proof to be absorbed into the displayed factors involving $1+\theta$.

    For each $x\in V(\mathcal H)$, fix a partition
    \[
        V(\mathcal H)\setminus\{x\}=V_1^x\cup\cdots\cup V_t^x
    \]
    maximizing $|L(x)[V_1^x,\ldots,V_t^x]|$. Define
    \[
        B^x:=L(x)\setminus L(x)[V_1^x,\ldots,V_t^x],
    \]
    and
    \[
        M^x:=K[V_1^x,\ldots,V_t^x]\setminus L(x).
    \]
    We call the edges in $B^x$ bad inside-part edges and the edges in $M^x$ missing cross-pairs.

    \begin{claim}\label{CLAIM:bad_missing_revised}
        For every $x\in V(\mathcal H)$,
        \[
            |B^x|\le \frac{1+\theta}{2}\varepsilon n^2
            \qquad\text{and}\qquad
            |M^x|\le (1+\theta)\varepsilon n^2 .
        \]
    \end{claim}

    \begin{proof}
        Since $L(x)$ is $K_{t+1}$-free, Theorem~\ref{THM:Furedi_linear_stability_new} applies to $L(x)$.  By~\eqref{EQ:min_degree_main_revised} and the trivial bound $e(T_{n-1,t})\le \frac12(1-\frac1t)n^2$, we have
        \[
            e(T_{n-1,t})-|L(x)|
            \le \frac{1+\theta}{2}\varepsilon n^2.
        \]
        F\"uredi's theorem says that $L(x)$ can be made $t$-partite by deleting at most this number of edges.  Since the chosen partition maximizes the number of cross-edges in the link, the number of inside-part link edges under this partition is at most the same quantity.  Therefore
        \[
            |B^x|\le \frac{1+\theta}{2}\varepsilon n^2.
        \]
        Moreover,
        \begin{align*}
            |M^x|
            = |K[V_1^x,\ldots,V_t^x]|-|L(x)[V_1^x,\ldots,V_t^x]|
            &\le e(T_{n-1,t})-\bigl(|L(x)|-|B^x|\bigr)\\
            &\le \frac{1+\theta}{2}\varepsilon n^2+\frac{1+\theta}{2}\varepsilon n^2
            =(1+\theta)\varepsilon n^2.
        \end{align*}
        This proves Claim~\ref{CLAIM:bad_missing_revised}.
    \end{proof}
    
    Next we record the part-size information that will be used in the weighted averaging. For a fixed vertex $x$, write
    \[
        x_i:=\frac{|V_i^x|}{n}\qquad (i\in[t]).
    \]

    \begin{claim}\label{CLAIM:L2_control_revised}
        For every $x\in V(\mathcal H)$,
        \[
            \sum_{i=1}^t\left(x_i-\frac1t\right)^2
            \le \frac{1+2\theta}{6t^2}.
        \]
    \end{claim}

    \begin{proof}
        We have
        \[
            |K[V_1^x,\ldots,V_t^x]|
            \ge |L(x)|-|B^x|.
        \]
        Using~\eqref{EQ:min_degree_main_revised} and Claim~\ref{CLAIM:bad_missing_revised},
        \[
            |K[V_1^x,\ldots,V_t^x]|
            \ge
            \frac12\left(1-\frac1t\right)n^2-(1+\theta)\varepsilon n^2.
        \]
        On the other hand,
        \[
            |K[V_1^x,\ldots,V_t^x]|
            =
            \frac12\left(\left(\sum_{i=1}^t x_i\right)^2-\sum_{i=1}^t x_i^2\right)n^2,
        \]
        and $\sum_i x_i=(n-1)/n$.  Since $n$ is sufficiently large, the terms coming from $(n-1)/n$ are absorbed into the additional $\theta\varepsilon n^2$ allowance, and we obtain
        \[
            \sum_{i=1}^t x_i^2
            \le \frac1t+2(1+2\theta)\varepsilon .
        \]
        Hence
        \[
            \sum_{i=1}^t\left(x_i-\frac1t\right)^2
            =
            \sum_{i=1}^t x_i^2-\frac{2}{t}\sum_{i=1}^t x_i+\frac1t
            \le 2(1+2\theta)\varepsilon
            =
            \frac{1+2\theta}{6t^2}.
        \]
        This proves Claim~\ref{CLAIM:L2_control_revised}.
    \end{proof}

    The following elementary consequence of Claim~\ref{CLAIM:L2_control_revised} will be used twice.

    \begin{claim}\label{CLAIM:elementary_size_bounds}
        For every fixed vertex $x$, with $x_i=|V_i^x|/n$, the following hold:
        \[
            \frac{(t-1)x_i}{1-x_i}\le \frac53
            \qquad\text{for all }i\in[t],
        \]
        and
        \[
            \sum_{i=1}^t \frac{x_i}{1-x_i}\le \frac53.
        \]
    \end{claim}

    \begin{proof}
        Let $b:=\max_i x_i$.  Since $\sum_i x_i=(n-1)/n$, Claim~\ref{CLAIM:L2_control_revised} implies, for $n$ sufficiently large, that
        \[
            b\le
            \frac1t+\sqrt{\frac{(1+3\theta)(t-1)}{6t^3}}.
        \]
        A straightforward calculation shows that, for every $t\ge3$ and $\theta=10^{-4}$,
        \[
            \frac1t+\sqrt{\frac{(1+3\theta)(t-1)}{6t^3}}
            \le \frac{5}{3t+2}.
        \]
        Therefore $x_i\le b\le 5/(3t+2)$ for every $i$, and so
        \[
            \frac{(t-1)x_i}{1-x_i}
            \le
            \frac{(t-1)\cdot \frac{5}{3t+2}}{1-\frac{5}{3t+2}}
            =
            \frac{5}{3}
        \]
        for every $i$.

        For the second inequality, use
        \[
            \frac{x_i}{1-x_i}
            =
            x_i+\frac{x_i^2}{1-x_i}
            \le
            x_i+\frac{x_i^2}{1-b}.
        \]
        Since $b\le 5/(3t+2)$, we have
        \[
            \frac1{1-b}\le \frac{3t+2}{3t-3}.
        \]
        Also, by Claim~\ref{CLAIM:L2_control_revised},
        \[
            \sum_i x_i^2\le \frac1t+\frac{1+2\theta}{6t^2}.
        \]
        Hence
        \[
            \sum_i\frac{x_i}{1-x_i}
            \le
            1+\left(\frac1t+\frac{1+2\theta}{6t^2}\right)\frac{3t+2}{3t-3}
            \le \frac53
        \]
        for every $t\ge3$ and $\theta=10^{-4}$.
    \end{proof}

    We now start the averaging argument.  Put $\Lambda:=\frac53$.  For $z\in V(\mathcal H)$, define
    \[
        \Phi(z):=\sum_{w\in V(\mathcal H)}
        \bigl(d_{M^w}(z)+\Lambda d_{B^w}(z)\bigr).
    \]
    By Claim~\ref{CLAIM:bad_missing_revised},
    \[
        \sum_{z\in V(\mathcal H)}\Phi(z)
        =
        2\sum_w |M^w|+2\Lambda\sum_w |B^w|
        \le
        (1+\theta)(2+\Lambda)\varepsilon n^3.
    \]
    Therefore, we can choose a vertex $x$ such that
    \begin{align}\label{EQ:choose_x_revised}
        \Phi(x)\le (1+\theta)(2+\Lambda)\varepsilon n^2.
    \end{align}
    Fix this vertex $x$ for the rest of the proof, and write
    \[
        V_j:=V_j^x,\qquad s_j:=|V_j|,\qquad x_j:=s_j/n.
    \]

    For each $j\in[t]$, define
    \[
        P_j:=\sum_{w\in V_j}d_{M^x}(w)+\sum_{w\in V_j}d_{B^w}(x),
    \]
    and
    \[
        Q_j:=\sum_{w\in V_j}d_{M^w}(x)+\sum_{w\in V_j}d_{B^x}(w).
    \]

    \begin{claim}\label{CLAIM:PQ_sum_revised}
        We have
        \[
            \sum_{j=1}^t(Q_j+\Lambda P_j)
            \le
            (1+\theta)(3\Lambda+3)\varepsilon n^2.
        \]
    \end{claim}

    \begin{proof}
        Expanding the left-hand side gives
        \begin{align*}
            \sum_j(Q_j+\Lambda P_j)
            &=
            2\Lambda |M^x|+2|B^x|
            +\sum_w d_{M^w}(x)+\Lambda\sum_w d_{B^w}(x).
        \end{align*}
        By Claim~\ref{CLAIM:bad_missing_revised} and~\eqref{EQ:choose_x_revised},
        \[
            2\Lambda |M^x|+2|B^x|
            \le (1+\theta)(2\Lambda+1)\varepsilon n^2
        \]
        and
        \[
            \sum_w d_{M^w}(x)+\Lambda\sum_w d_{B^w}(x)
            \le (1+\theta)(2+\Lambda)\varepsilon n^2.
        \]
        Adding these two bounds proves the claim.
    \end{proof}

    Next we choose a second part for each possible $j$.  Fix $j\in[t]$.  For $i\ne j$, define
    \[
        C_{i,j}:=\sum_{w\in V_j}\sum_{a\in V_i} d_{M^w}(a).
    \]
    By Claim~\ref{CLAIM:bad_missing_revised},
    \begin{align}\label{EQ:C_sum_revised}
        \sum_{i\ne j}C_{i,j}
        \le
        \sum_{w\in V_j}\sum_{a\in V(\mathcal H)}d_{M^w}(a)
        =
        2\sum_{w\in V_j}|M^w|
        \le
        2(1+\theta)\varepsilon n^2s_j.
    \end{align}
    Averaging over $i\ne j$ with weights $|V_i|$, and using
    \[
        \sum_{i\ne j}|V_i|=n-1-s_j,
    \]
    we find an index $i=i(j)\ne j$ such that
    \begin{align}\label{EQ:weighted_i_choice_revised}
        \frac{s_j}{|V_i|}P_j+\frac{C_{i,j}}{|V_i|}
        &\le
        \frac{(t-1)s_j}{n-1-s_j}P_j
        +
        \frac{2(1+\theta)\varepsilon n^2s_j}{n-1-s_j}.
    \end{align}
    Since $n$ is sufficiently large, the right-hand side is at most
    \[
        \frac{(t-1)x_j}{1-x_j}P_j
        +
        2(1+\theta)\varepsilon n^2\frac{x_j}{1-x_j}.
    \]
    By Claim~\ref{CLAIM:elementary_size_bounds},
    \[
        \frac{(t-1)x_j}{1-x_j}\le \Lambda.
    \]
    Hence
    \begin{align}\label{EQ:weighted_i_choice_final}
        Q_j+\frac{s_j}{|V_i|}P_j+\frac{C_{i,j}}{|V_i|}
        \le
        Q_j+\Lambda P_j
        +
        2(1+\theta)\varepsilon n^2\frac{x_j}{1-x_j}.
    \end{align}

    Summing~\eqref{EQ:weighted_i_choice_final} over $j$ and using Claims~\ref{CLAIM:PQ_sum_revised} and~\ref{CLAIM:elementary_size_bounds}, we obtain
    \begin{align}\label{EQ:error_total_revised}
        \sum_{j=1}^t
        \left(
            Q_j+\frac{s_j}{|V_{i(j)}|}P_j+\frac{C_{i(j),j}}{|V_{i(j)}|}
        \right)
        \le
        (1+\theta)K\varepsilon n^2,
    \end{align}
    where
    \[
        K:=3\Lambda+3+\frac{10}{3}=\frac{34}{3}.
    \]

    We now choose the part $B$ by averaging with weights $s_j^2$.  Since
    \[
        \sum_{j=1}^t s_j^2\ge \frac{(n-1)^2}{t},
    \]
    it follows from~\eqref{EQ:error_total_revised} that there is some $j\in[t]$ such that, writing $i=i(j)$, $A:=V_i$, $B:=V_j$, $r:=|A|$, and $s:=|B|$, we have
    \begin{align}\label{EQ:chosen_ij_error_revised}
        Q_j+\frac{s}{r}P_j+\frac{C_{i,j}}{r}
        \le
        (1+\theta)K\varepsilon n^2\cdot \frac{s^2}{\sum_\ell s_\ell^2}
        \le
        (1+\theta)K\varepsilon t s^2.
    \end{align}

    For each $w\in B$, let $P(w)$ be the part of the chosen partition $V_1^w\cup\cdots\cup V_t^w$ that contains $x$.

    \begin{claim}\label{CLAIM:local_one_revised}
        For every $w\in B$,
        \[
            |P(w)\cap A|\le d_{M^x}(w)+d_{B^w}(x).
        \]
    \end{claim}

    \begin{proof}
        Let $a\in P(w)\cap A$.  Since $a\in A=V_i^x$, $w\in B=V_j^x$, and $i\ne j$, the pair $aw$ is a cross-pair in the $x$-partition.  If $xaw\notin\mathcal H$, then $aw\in M^x$.  If $xaw\in\mathcal H$, then $xa$ is an edge of $L(w)$ whose two endpoints lie in the same part $P(w)$, and hence $xa\in B^w$.  This gives an injection from $P(w)\cap A$ into the union of the two sets counted by $d_{M^x}(w)$ and $d_{B^w}(x)$, proving the claim.
    \end{proof}

    \begin{claim}\label{CLAIM:local_two_revised}
        For every $w\in B$,
        \[
            |B\setminus P(w)|\le d_{M^w}(x)+d_{B^x}(w)+1.
        \]
    \end{claim}

    \begin{proof}
        Let $b\in B\setminus(P(w)\cup\{w\})$.  Since $x\in P(w)$ and $b\notin P(w)$, the pair $xb$ is a cross-pair in the $w$-partition.  If $xbw\notin\mathcal H$, then $xb\in M^w$.  If $xbw\in\mathcal H$, then $bw$ is an edge of $L(x)$ whose two endpoints both lie in $B=V_j^x$, so $bw\in B^x$.  The additional $+1$ accounts for the possible vertex $w$ itself.
    \end{proof}

    For $a\in A$, define
    \[
        T(a):=\sum_{\substack{w\in B\\ a\notin P(w)}} |N_{L(w)}(a)\cap B|.
    \]
    We estimate the average value of $T(a)$.  Fix $w\in B$ and $a\in A\setminus P(w)$. If $b\in B\cap P(w)$, then $ab$ is a cross-pair in the $w$-partition.  Therefore the only reason for such a vertex $b$ not to lie in $N_{L(w)}(a)$ is that $ab\in M^w$. Hence
    \[
        |N_{L(w)}(a)\cap B|
        \ge |B\cap P(w)|-d_{M^w}(a).
    \]
    By Claim~\ref{CLAIM:local_two_revised},
    \[
        |B\cap P(w)|\ge s-1-d_{M^w}(x)-d_{B^x}(w).
    \]
    Thus
    \begin{align}\label{EQ:pointwise_T_revised}
        |N_{L(w)}(a)\cap B|
        \ge s-1-d_{M^w}(x)-d_{B^x}(w)-d_{M^w}(a).
    \end{align}

    Let
    \[
        r_w:=|A\cap P(w)|.
    \]
    By Claim~\ref{CLAIM:local_one_revised},
    \[
        r_w\le d_{M^x}(w)+d_{B^w}(x).
    \]
    Summing~\eqref{EQ:pointwise_T_revised} over all pairs $(w,a)$ with $w\in B$ and $a\in A\setminus P(w)$ gives
    \begin{align*}
        \sum_{a\in A}T(a)
        &\ge
        \sum_{w\in B}(r-r_w)
        \bigl(s-1-d_{M^w}(x)-d_{B^x}(w)\bigr)
        -
        \sum_{w\in B}\sum_{a\in A}d_{M^w}(a)\\
        &\ge
        rs(s-1)
        -s\sum_{w\in B}r_w
        -r\sum_{w\in B}\bigl(d_{M^w}(x)+d_{B^x}(w)\bigr)
        -C_{i,j}\\
        &\ge
        rs(s-1)-sP_j-rQ_j-C_{i,j}.
    \end{align*}
    Therefore there exists $a_0\in A$ such that
    \[
        T(a_0)
        \ge
        s(s-1)-\frac{s}{r}P_j-Q_j-\frac{C_{i,j}}{r}.
    \]
    Using~\eqref{EQ:chosen_ij_error_revised}, we get
    \begin{align}\label{EQ:T_lower_revised}
        T(a_0)
        \ge
        s^2-s-(1+\theta)K\varepsilon t s^2.
    \end{align}

    Every contribution to $T(a_0)$ is an ordered pair $(w,b)\in B^2$ such that $a_0bw\in\mathcal H$.  Therefore $bw$ is an edge of $L(a_0)[B]$, and each unordered edge is counted at most twice.  Hence
    \begin{align}\label{EQ:link_lower_revised}
        |L(a_0)[B]|
        \ge
        \frac12\left(s^2-s-(1+\theta)K\varepsilon t s^2\right).
    \end{align}
    On the other hand, $L(a_0)$ is $K_{t+1}$-free.  By Tur\'an's theorem,
    \begin{align}\label{EQ:link_upper_revised}
        |L(a_0)[B]|
        \le
        \left(1-\frac1t\right)\frac{s^2}{2}.
    \end{align}
    Comparing~\eqref{EQ:link_lower_revised} and~\eqref{EQ:link_upper_revised}, we obtain
    \begin{align}\label{EQ:final_compare_revised}
        (1+\theta)K\varepsilon t s^2\ge \frac{s^2}{t}-s.
    \end{align}

    Since Claim~\ref{CLAIM:L2_control_revised} implies $s\ge c_t n$ for some positive constant $c_t$ depending only on $t$, the term $s$ is smaller than $\frac12\left(1-(1+\theta)K/12\right)s^2/t$ for all sufficiently large $n$. But $\varepsilon=1/(12t^2)$, and therefore
    \[
        (1+\theta)K\varepsilon t
        =
        \frac{(1+\theta)K}{12t}
        <
        \frac1t
    \]
    by the choice of $\theta$.  This contradicts~\eqref{EQ:final_compare_revised}.  The proof of Theorem~\ref{THM:daisy_Turan_upper_bound} is complete.
\end{proof}

\section{Proof of Theorem~\ref{THM:positive_upper}}\label{SEC:PositiveCodegree}
In this section we prove Theorem~\ref{THM:positive_upper}. We use the following classical theorem of Andr\'asfai--Erd\H{o}s--S\'os.

\begin{theorem}[\cite{AES74}]\label{THM:AES}
    Suppose that $G$ is a $K_{t+1}$-free graph on $m$ vertices with $\delta(G)>\frac{3t-4}{3t-1}m$. Then $G$ is $t$-partite.
\end{theorem}

\begin{proof}[Proof of Theorem~\ref{THM:positive_upper}]
    We first record the standard reduction. Let $r\ge3$, and let $\mathcal G$ be an $(r+1)$-graph on $n$ vertices which is $\mathcal D_{r+1,t+1}$-free and satisfies
    \[
        \delta_r^+(\mathcal G)=\alpha n.
    \]
    Choose a vertex $v$ of positive degree and consider the $r$-graph
    \[
        \mathcal L:=L_{\mathcal G}(v).
    \]
    Then $\mathcal L$ is $\mathcal D_{r,t+1}$-free. Indeed, a copy of $\mathcal D_{r,t+1}$ in $\mathcal L$ would extend with $v$ to a copy of $\mathcal D_{r+1,t+1}$ in $\mathcal G$. Moreover, if an $(r-1)$-set $S$ has positive degree in $\mathcal L$, then $S\cup\{v\}$ has positive degree in $\mathcal G$, and hence
    \[
        d_{\mathcal L}(S)=d_{\mathcal G}(S\cup\{v\})\ge \alpha n.
    \]
    Thus
    \[
        \delta_{r-1}^+(\mathcal L)\ge \alpha n
        =
        (\alpha+o(1))|V(\mathcal L)|.
    \]
    It follows that
    \[
        \gamma_r^+(\mathcal D_{r+1,t+1})
        \le
        \gamma_{r-1}^+(\mathcal D_{r,t+1}).
    \]
    Iterating this inequality, it suffices to prove Theorem~\ref{THM:positive_upper} for the case $r=3$.

    Put $s:=1-\frac1t$ and let $\mathcal H$ be an $n$-vertex $\mathcal D_{3,t+1}$-free $3$-graph with $\delta_2^+(\mathcal H)=\alpha n$. We prove that
    \[
        \alpha\le s-\frac{1}{36t^2}+o(1).
    \]

    Suppose instead that
    \[
        \alpha>s-\frac{1}{36t^2}.
    \]
    Since
    \[
        s-\frac{3t-4}{3t-1}=\frac{1}{t(3t-1)}>\frac{1}{36t^2}
    \]
    for every $t\ge2$, we have
    \[
        \alpha>\frac{3t-4}{3t-1}.
    \]

    Fix a vertex $v$ with positive degree, and let
    \[
        U_v:=\{u\in V(\mathcal H): d_{\mathcal H}(uv)>0\}.
    \]
    The graph $L(v)[U_v]$ is $K_{t+1}$-free, and every vertex of it has degree at least $\alpha n$. Since $|U_v|\le n$, Theorem~\ref{THM:AES} implies that $L(v)[U_v]$ is $t$-partite. Adding isolated vertices to arbitrary parts, we conclude that every link $L(v)$ is $t$-partite.

    By Theorem~\ref{THM:t_partite_upper},
    \begin{align}\label{EQ:edge_upper_from_partite}
        |\mathcal H|
        \le \left(s-\frac{1}{12t^2}+o(1)\right)\binom n3.
    \end{align}

    We now derive a lower bound on $|\mathcal H|$ from positive codegree. If $v$ is non-isolated, then $L(v)[U_v]$ is $K_{t+1}$-free and has minimum degree at least $\alpha n$. Since a $t$-partite graph on $m$ vertices has minimum degree at most $sm$, we have
    \[
        |U_v|\ge \frac{\alpha}{s}n.
    \]
    Consequently,
    \[
        |L(v)|\ge \frac12 |U_v|\alpha n
        \ge \frac{\alpha^2}{2s}n^2.
    \]
    Moreover, any non-isolated vertex $v$ has at least $|U_v|\ge (\alpha/s)n$ neighbours in the shadow, and these vertices are also non-isolated. Thus $\mathcal H$ has at least $(\alpha/s)n$ non-isolated vertices. Hence
    \[
        3|\mathcal H|=\sum_{v\in V(\mathcal H)}|L(v)|
        \ge \frac{\alpha}{s}n\cdot \frac{\alpha^2}{2s}n^2
        =\frac{\alpha^3}{2s^2}n^3.
    \]
    Dividing by $\binom n3$, we get
    \begin{align}\label{EQ:edge_lower_from_codeg}
        \frac{|\mathcal H|}{\binom n3}\ge \frac{\alpha^3}{s^2}-o(1).
    \end{align}

    Let $x:=\frac{1}{36t^2}$. Since $\alpha>s-x$,
    \[
        \frac{\alpha^3}{s^2}>\frac{(s-x)^3}{s^2}
        =s-3x+\frac{3x^2}{s}-\frac{x^3}{s^2}.
    \]
    But $3x=1/(12t^2)$ and
    \[
        \frac{3x^2}{s}-\frac{x^3}{s^2}=\frac{x^2}{s^2}(3s-x)>0.
    \]
    Therefore
    \[
        \frac{\alpha^3}{s^2}>s-\frac{1}{12t^2}.
    \]
    This contradicts~\eqref{EQ:edge_upper_from_partite} and~\eqref{EQ:edge_lower_from_codeg}. Hence
    \[
        \gamma_2^+(\mathcal D_{3,t+1})\le 1-\frac1t-\frac{1}{36t^2}.
    \]
    This proves the case $r=3$, and the general case follows from the reduction above.
\end{proof}

\section*{Acknowledgments}
J.H. was supported by the National Key R\&D Program of China (No.~2023YFA1010202) and by the Central Guidance on Local Science and Technology Development Fund of Fujian Province (No.~2023L3003). 
X.H. was supported by the National Key Research and Development Program of China (2023YFA1010203), the National Natural Science Foundation of China (12471336), and Quantum Science
and Technology-National Science and Technology Major Project (2021ZD0302902).
X.L. was supported by the Excellent Young Talents Program (Overseas) of the National Natural Science Foundation of China.

\section*{Declaration on the use of AI}
The authors used generative AI tools to assist with discussing proof strategies, proof checking, and exposition. All mathematical arguments, results, and conclusions were reviewed and verified by the authors.
\bibliographystyle{alpha}
\bibliography{daisy}

@article {Pik23,
    AUTHOR = {Pikhurko, Oleg},
     TITLE = {On the limit of the positive {$\ell$}-degree {T}ur\'an problem},
   JOURNAL = {Electron. J. Combin.},
  FJOURNAL = {Electronic Journal of Combinatorics},
    VOLUME = {30},
      YEAR = {2023},
    NUMBER = {3},
     PAGES = {Paper No. 3.25, 15},
      ISSN = {1077-8926},
   MRCLASS = {05D05 (05C65)},
  MRNUMBER = {4635475},
MRREVIEWER = {Shao-qiang\ Liu},
       DOI = {10.37236/11912},
       URL = {https://doi.org/10.37236/11912},
}

@article {FV13,
    AUTHOR = {Falgas-Ravry, Victor and Vaughan, Emil R.},
     TITLE = {{A}pplications of the semi-definite method to the {T}ur{\'{a}}n density problem for $3$-graphs},
   JOURNAL = {Combin. Probab. Comput.},
  FJOURNAL = {Combinatorics, Probability and Computing},
    VOLUME = {22},
      YEAR = {2013},
    NUMBER = {1},
     PAGES = {21--54},
      ISSN = {0963-5483,1469-2163},
   MRCLASS = {05D05 (05C65)},
  MRNUMBER = {3002572},
       DOI = {10.1017/S0963548312000508},
       URL = {https://doi.org/10.1017/S0963548312000508},
}

@article {JDM60,
    AUTHOR = {Johnson, Diane M. and Dulmage, A. L. and Mendelsohn, N. S.},
     TITLE = {On an algorithm of {G}. {B}irkhoff concerning doubly stochastic matrices},
   JOURNAL = {Canad. Math. Bull.},
  FJOURNAL = {Canadian Mathematical Bulletin. Bulletin Canadien de  Math\'ematiques},
    VOLUME = {3},
      YEAR = {1960},
     PAGES = {237--242},
      ISSN = {0008-4395,1496-4287},
   MRCLASS = {15.65},
  MRNUMBER = {130267},
MRREVIEWER = {Marvin\ Marcus},
       DOI = {10.4153/CMB-1960-029-5},
       URL = {https://doi.org/10.4153/CMB-1960-029-5},
}

@article{Cra36,
  title={On the order of magnitude of the difference between consecutive prime numbers},
  author={Cram{\'e}r, Harald},
  journal={Acta arithmetica},
  volume={2},
  pages={23--46},
  year={1936},
  publisher={Instytut Matematyczny Polskiej Akademii Nauk}
}

@article {Bir46,
    AUTHOR = {Birkhoff, Garrett},
     TITLE = {Three observations on linear algebra},
   JOURNAL = {Univ. Nac. Tucum\'an. Revista A.},
  FJOURNAL = {Univ. Nac. Tucum\'an. Revista A.},
    VOLUME = {5},
      YEAR = {1946},
     PAGES = {147--151},
   MRCLASS = {09.1X},
  MRNUMBER = {20547},
MRREVIEWER = {J.\ L.\ Dorroh},
}

@article {AKS07,
    AUTHOR = {Alon, Noga and Krech, Anja and Szab{\'o}, Tibor},
     TITLE = {Tur{\'a}n's theorem in the hypercube},
   JOURNAL = {SIAM J. Discrete Math.},
  FJOURNAL = {SIAM Journal on Discrete Mathematics},
    VOLUME = {21},
      YEAR = {2007},
    NUMBER = {1},
     PAGES = {66--72},
      ISSN = {0895-4801,1095-7146},
   MRCLASS = {05D05 (05C35 05C55 05D10)},
  MRNUMBER = {2299695},
MRREVIEWER = {Yi\ Zhao},
       DOI = {10.1137/060649422},
       URL = {https://doi.org/10.1137/060649422},
}

@incollection {GL16,
    AUTHOR = {Griggs, Jerrold R. and Li, Wei-Tian},
     TITLE = {Progress on poset-free families of subsets},
 BOOKTITLE = {Recent trends in combinatorics},
    SERIES = {IMA Vol. Math. Appl.},
    VOLUME = {159},
     PAGES = {317--338},
 PUBLISHER = {Springer, [Cham]},
      YEAR = {2016},
      ISBN = {978-3-319-24296-5; 978-3-319-24298-9},
   MRCLASS = {05-02 (05D05 06A07)},
  MRNUMBER = {3526415},
       DOI = {10.1007/978-3-319-24298-9\_14},
       URL = {https://doi.org/10.1007/978-3-319-24298-9_14},
}

@article {BLM11,
    AUTHOR = {Bollob{\'a}s, B{\'e}la and Leader, Imre and Malvenuto, Claudia},
     TITLE = {Daisies and other {T}ur{\'a}n problems},
   JOURNAL = {Combin. Probab. Comput.},
  FJOURNAL = {Combinatorics, Probability and Computing},
    VOLUME = {20},
      YEAR = {2011},
    NUMBER = {5},
     PAGES = {743--747},
      ISSN = {0963-5483,1469-2163},
   MRCLASS = {05D05 (05C35 05C65)},
  MRNUMBER = {2825587},
MRREVIEWER = {Ivan\ Pashov},
       DOI = {10.1017/S0963548311000319},
       URL = {https://doi.org/10.1017/S0963548311000319},
}

@article {EIL24,
    AUTHOR = {Ellis, David and Ivan, Maria-Romina and Leader, Imre},
     TITLE = {Tur{\'a}n densities for daisies and hypercubes},
   JOURNAL = {Bull. Lond. Math. Soc.},
  FJOURNAL = {Bulletin of the London Mathematical Society},
    VOLUME = {56},
      YEAR = {2024},
    NUMBER = {12},
     PAGES = {3838--3853},
      ISSN = {0024-6093,1469-2120},
   MRCLASS = {05D05},
  MRNUMBER = {4835279},
MRREVIEWER = {William\ Linz},
       DOI = {10.1112/blms.13171},
       URL = {https://doi.org/10.1112/blms.13171},
}

@article {EIL25,
    AUTHOR = {Ellis, David and Ivan, Maria-Romina and Leader, Imre},
     TITLE = {{T}ur{\'a}n densities for small hypercubes},
   JOURNAL = {SIAM J. Discrete Math.},
  FJOURNAL = {SIAM Journal on Discrete Mathematics},
    VOLUME = {39},
      YEAR = {2025},
    NUMBER = {2},
     PAGES = {1363--1371},
      ISSN = {0895-4801,1095-7146},
   MRCLASS = {05C65},
  MRNUMBER = {4922383},
MRREVIEWER = {J\'ozsef\ Balogh},
       DOI = {10.1137/24M1710413},
       URL = {https://doi.org/10.1137/24M1710413},
}

@article {JT10,
    AUTHOR = {Johnson, J. Robert and Talbot, John},
     TITLE = {{V}ertex {T}ur{\'a}n problems in the hypercube},
   JOURNAL = {J. Combin. Theory Ser. A},
  FJOURNAL = {Journal of Combinatorial Theory. Series A},
    VOLUME = {117},
      YEAR = {2010},
    NUMBER = {4},
     PAGES = {454--465},
      ISSN = {0097-3165,1096-0899},
   MRCLASS = {05C35},
  MRNUMBER = {2592894},
MRREVIEWER = {Maria\ A.\ Axenovich},
       DOI = {10.1016/j.jcta.2009.07.004},
       URL = {https://doi.org/10.1016/j.jcta.2009.07.004},
}

@article {Furedi15,
    AUTHOR = {F{\"u}redi, Zolt{\'a}n},
     TITLE = {A proof of the stability of extremal graphs, {S}imonovits'   stability from {S}zemer{\'e}di's regularity},
   JOURNAL = {J. Combin. Theory Ser. B},
  FJOURNAL = {Journal of Combinatorial Theory. Series B},
    VOLUME = {115},
      YEAR = {2015},
     PAGES = {66--71},
      ISSN = {0095-8956,1096-0902},
   MRCLASS = {05C35},
  MRNUMBER = {3383250},
MRREVIEWER = {Yi\ Zhao},
       DOI = {10.1016/j.jctb.2015.05.001},
       URL = {https://doi.org/10.1016/j.jctb.2015.05.001},
}

@article {BT11,
    AUTHOR = {Baber, Rahil and Talbot, John},
     TITLE = {{H}ypergraphs do jump},
   JOURNAL = {Combin. Probab. Comput.},
  FJOURNAL = {Combinatorics, Probability and Computing},
    VOLUME = {20},
      YEAR = {2011},
    NUMBER = {2},
     PAGES = {161--171},
      ISSN = {0963-5483,1469-2163},
   MRCLASS = {05C65 (05C35)},
  MRNUMBER = {2769186},
MRREVIEWER = {Yi\ Zhao},
       DOI = {10.1017/S0963548310000222},
       URL = {https://doi.org/10.1017/S0963548310000222},
}

@article {Mu03,
    AUTHOR = {Mubayi, Dhruv},
     TITLE = {{O}n hypergraphs with every four points spanning at most two triples},
   JOURNAL = {Electron. J. Combin.},
  FJOURNAL = {Electronic Journal of Combinatorics},
    VOLUME = {10},
      YEAR = {2003},
     PAGES = {Note 10, 4},
      ISSN = {1077-8926},
   MRCLASS = {05C35 (05C65 05D05)},
  MRNUMBER = {2014538},
       DOI = {10.37236/1750},
       URL = {https://doi.org/10.37236/1750},
}

@article {Cae83,
    AUTHOR = {de Caen, D.},
     TITLE = {Extension of a theorem of {M}oon and {M}oser on complete  subgraphs},
   JOURNAL = {Ars Combin.},
  FJOURNAL = {Ars Combinatoria},
    VOLUME = {16},
      YEAR = {1983},
     PAGES = {5--10},
      ISSN = {0381-7032},
   MRCLASS = {05C30},
  MRNUMBER = {734038},
MRREVIEWER = {E.\ M.\ Palmer},
}

@article {Tur41,
    AUTHOR = {Tur{\'a}n, Paul},
     TITLE = {Eine {E}xtremalaufgabe aus der {G}raphentheorie},
   JOURNAL = {Mat. Fiz. Lapok},
  FJOURNAL = {Matematikai {\'{e}}s Fizikai Lapok},
    VOLUME = {48},
      YEAR = {1941},
     PAGES = {436--452},
      ISSN = {0302-7317},
   MRCLASS = {56.0X},
  MRNUMBER = {18405},
MRREVIEWER = {P.\ Erd{\H{o}}s},
}

@article {ES46,
    AUTHOR = {Erd{\H o}s, P. and Stone, A. H.},
     TITLE = {On the structure of linear graphs},
   JOURNAL = {Bull. Amer. Math. Soc.},
  FJOURNAL = {Bulletin of the American Mathematical Society},
    VOLUME = {52},
      YEAR = {1946},
     PAGES = {1087--1091},
      ISSN = {0002-9904},
   MRCLASS = {56.0X},
  MRNUMBER = {18807},
MRREVIEWER = {H.\ S. M. Coxeter},
       DOI = {10.1090/S0002-9904-1946-08715-7},
       URL = {https://doi.org/10.1090/S0002-9904-1946-08715-7},
}

@article {ES66,
    AUTHOR = {Erd{\H o}s, P. and Simonovits, M.},
     TITLE = {A limit theorem in graph theory},
   JOURNAL = {Studia Sci. Math. Hungar.},
  FJOURNAL = {Studia Scientiarum Mathematicarum Hungarica. Combinatorics,
              Geometry and Topology (CoGeTo)},
    VOLUME = {1},
      YEAR = {1966},
     PAGES = {51--57},
      ISSN = {0081-6906,1588-2896},
   MRCLASS = {05.40},
  MRNUMBER = {205876},
MRREVIEWER = {W.\ Moser},
}

@incollection {Kee11,
    AUTHOR = {Keevash, Peter},
     TITLE = {Hypergraph {T}ur{\'a}n problems},
 BOOKTITLE = {Surveys in combinatorics 2011},
    SERIES = {London Math. Soc. Lecture Note Ser.},
    VOLUME = {392},
     PAGES = {83--139},
 PUBLISHER = {Cambridge Univ. Press, Cambridge},
      YEAR = {2011},
      ISBN = {978-1-107-60109-3},
   MRCLASS = {05-02 (05C65)},
  MRNUMBER = {2866732},
}

@article {FF84,
    AUTHOR = {Frankl, P. and F{\"u}redi, Z.},
     TITLE = {An exact result for {$3$}-graphs},
   JOURNAL = {Discrete Math.},
  FJOURNAL = {Discrete Mathematics},
    VOLUME = {50},
      YEAR = {1984},
    NUMBER = {2-3},
     PAGES = {323--328},
      ISSN = {0012-365X,1872-681X},
   MRCLASS = {05C35 (05C65)},
  MRNUMBER = {753720},
MRREVIEWER = {Ralph\ Faudree},
       DOI = {10.1016/0012-365X(84)90058-X},
       URL = {https://doi.org/10.1016/0012-365X(84)90058-X},
}

@article {AES74,
    AUTHOR = {Andr\'{a}sfai, B. and Erd\H{o}s, P. and S\'{o}s, V. T.},
     TITLE = {On the connection between chromatic number, maximal clique and minimal degree of a graph},
   JOURNAL = {Discrete Math.},
  FJOURNAL = {Discrete Mathematics},
    VOLUME = {8},
      YEAR = {1974},
     PAGES = {205--218},
      ISSN = {0012-365X,1872-681X},
   MRCLASS = {05C15},
  MRNUMBER = {340075},
MRREVIEWER = {D.\ J.\ Kleitman},
       DOI = {10.1016/0012-365X(74)90133-2},
       URL = {https://doi.org/10.1016/0012-365X(74)90133-2},
}

@inproceedings{MPS11,
  title={{H}ypergraph {T}ur{\'a}n problem: some open questions},
  author={Mubayi, Dhruv and Pikhurko, Oleg and Sudakov, Benny},
  booktitle={AIM workshop problem lists, manuscript},
  pages={166},
  year={2011}
}

@article {BHP01,
    AUTHOR = {Baker, R. C. and Harman, G. and Pintz, J.},
     TITLE = {The difference between consecutive primes. {II}},
   JOURNAL = {Proc. London Math. Soc. (3)},
  FJOURNAL = {Proceedings of the London Mathematical Society. Third Series},
    VOLUME = {83},
      YEAR = {2001},
    NUMBER = {3},
     PAGES = {532--562},
      ISSN = {0024-6115,1460-244X},
   MRCLASS = {11N05 (11N36)},
  MRNUMBER = {1851081},
MRREVIEWER = {D.\ R.\ Heath-Brown},
       DOI = {10.1112/plms/83.3.532},
       URL = {https://doi.org/10.1112/plms/83.3.532},
}

@article {Kos76,
    AUTHOR = {Kosto{\v c}ka, E. A.},
     TITLE = {Piercing the edges of the {$n$}-dimensional unit cube},
   JOURNAL = {Diskret. Analiz},
  FJOURNAL = {Akademiya Nauk SSSR. Sibirskoe Otdelenie. Institut Matematiki. Diskretny\u i\ Analiz. Sbornik Trudov},
      YEAR = {1976},
     PAGES = {55--64, 79},
   MRCLASS = {52A25},
  MRNUMBER = {467534},
MRREVIEWER = {Z.\ N\'aden\'ik},
}

@article {JE89,
    AUTHOR = {Johnson, Karen Anne and Entringer, Roger},
     TITLE = {Largest induced subgraphs of the {$n$}-cube that contain no {$4$}-cycles},
   JOURNAL = {J. Combin. Theory Ser. B},
  FJOURNAL = {Journal of Combinatorial Theory. Series B},
    VOLUME = {46},
      YEAR = {1989},
    NUMBER = {3},
     PAGES = {346--355},
      ISSN = {0095-8956,1096-0902},
   MRCLASS = {05C35 (05C38)},
  MRNUMBER = {999702},
MRREVIEWER = {S.\ F.\ Kapoor},
       DOI = {10.1016/0095-8956(89)90054-3},
       URL = {https://doi.org/10.1016/0095-8956(89)90054-3},
}

@article{HLZ26DaisyOld,
    AUTHOR = {Hou, Jianfeng and Liu, Xizhi and Zhang, Yixiao},
     TITLE = {Improved upper bound for the {Tur\' a}n density of daisies},
      NOTE = {Manuscript},
      YEAR = {2026},
}

@article{HLP24,
    AUTHOR = {Halfpap, Anastasia and Lemons, Nathan and Palmer, Cory},
     TITLE = {Positive co-degree density of hypergraphs},
   JOURNAL = {arXiv preprint arXiv:2207.05639},
      YEAR = {2024},
}

@article{BHLP24,
    AUTHOR = {Balogh, J\'{o}zsef and Halfpap, Anastasia and Lidick\'{y}, Bernard and Palmer, Cory},
     TITLE = {{P}ositive co-degree densities and jumps},
   JOURNAL = {arXiv preprint arXiv:2412.08597},
      YEAR = {2024},
}

@article {Alo24,
AUTHOR = {Alon, Noga},
TITLE = {Erasure list-decodable codes and {T}ur{\'a}n hypercube problems},
JOURNAL = {Finite Fields Appl.},
VOLUME = {100},
YEAR = {2024},
PAGES = {Paper No. 102513, 15},
DOI = {10.1016/j.ffa.2024.102513},
}

@article {Buk09,
AUTHOR = {Bukh, Boris},
TITLE = {Set families with a forbidden subposet},
JOURNAL = {Electron. J. Combin.},
VOLUME = {16},
YEAR = {2009},
NUMBER = {1},
PAGES = {Research Paper 142, 11},
DOI = {10.37236/231},
}

@article {Hal35,
AUTHOR = {Hall, P.},
TITLE = {On {R}epresentatives of {S}ubsets},
JOURNAL = {J. London Math. Soc.},
VOLUME = {10},
YEAR = {1935},
NUMBER = {1},
PAGES = {26--30},
DOI = {10.1112/jlms/s1-10.37.26},
}
\end{document}